\documentclass[12pt]{article}
\usepackage[T1]{fontenc}
\usepackage{lmodern}
\usepackage{amsmath,amssymb,amsthm}
\usepackage{array,booktabs,tabularx}
\usepackage{mathrsfs}
\usepackage{enumitem}
\usepackage{microtype}
\usepackage{xcolor}
\usepackage{needspace}
\usepackage[colorlinks=true,allcolors=blue,backref=page]{hyperref}
\hypersetup{
 citecolor=blue,
 pdftitle={Nonnegativity of the g-polynomial of split matroids},
 pdfauthor={Alice L. L. Gao and Matthew H. Y. Xie},
 pdfsubject={Combinatorics and matroid theory},
 pdfkeywords={split matroid, cyclic flat, Speyer g-polynomial}
}

\newtheorem{theorem}{Theorem}[section]
\newtheorem{proposition}[theorem]{Proposition}
\newtheorem{lemma}[theorem]{Lemma}
\newtheorem{corollary}[theorem]{Corollary}

\theoremstyle{definition}
\newtheorem{definition}[theorem]{Definition}
\theoremstyle{remark}
\newtheorem{remark}[theorem]{Remark}
\newtheorem{example}[theorem]{Example}

\DeclareMathOperator{\rk}{rk}
\newcommand{\M}{\mathsf{M}}
\newcommand{\N}{\mathsf{N}}
\newcommand{\U}{\mathsf{U}}
\newcommand{\cC}{\mathcal{C}}
\newcommand{\cD}{\mathcal{D}}
\newcommand{\cK}{\mathcal{K}}
\newcommand{\cR}{\mathcal{R}}
\newcommand{\LM}{\mathsf{LM}}
\newcommand{\La}{\mathsf{\Lambda}}
\newcommand{\Nstep}{\mathsf{N}}
\newcommand{\Estep}{\mathsf{E}}

\makeatletter
\renewcommand{\section}{%
  \@startsection{section}{1}{\z@}%
  {-3.5ex \@plus -1ex \@minus -.2ex}%
  {2.3ex \@plus .2ex}%
  {\normalfont\fontfamily{cmr}\fontseries{bx}\Large\selectfont}}
\renewcommand{\subsection}{%
  \@startsection{subsection}{2}{\z@}%
  {-3.25ex \@plus -1ex \@minus -.2ex}%
  {1.5ex \@plus .2ex}%
  {\normalfont\fontfamily{cmr}\fontseries{bx}\large\selectfont}}

\makeatother

\begin{document}

\begin{center}
{\large\bfseries
Nonnegativity of the $g$-polynomial of split matroids
}
\end{center}

\begin{center}
		Alice L. L. Gao$^1$ and
Matthew H. Y. Xie$^2$ \\[6pt]

        $^{1}$School of Mathematics and Statistics,\\
		Northwestern Polytechnical University, Xi'an, Shaanxi 710072, P.R. China\\[6pt]
		$^{2}$
        School of Mathematical Sciences, \\
        Tianjin University of Technology,
Tianjin 300384, P. R. China\\[6pt]

		Email: $^{1}$\texttt{llgao@nwpu.edu.cn},
			   $^{2}$\texttt{xie@email.tjut.edu.cn}
		\end{center}
\noindent\textbf{Abstract.}
We prove that the $g$-polynomial of every split matroid has nonnegative
coefficients, establishing Speyer's conjecture for a class closed under
taking minors and containing all paving and copaving matroids. Our proof
uses a deletion--contraction identity obtained by constructing an auxiliary
split matroid. We first show that every simple, cosimple, connected split
matroid $\M$ has an element $e$ for which both the deletion
$\M\setminus e$ and the contraction $\M/e$ are connected. More generally,
let $\M$ be any connected split matroid of rank $k$ on a ground set $E$
with $|E|\ge4$. Suppose that $e\in E$ is such that
both $\M\setminus e$ and $\M/e$ are connected. For every
$v\in E\setminus\{e\}$, we construct a connected
elementary split matroid $\N_{e,v}$ of rank $k-1$ on
$E\setminus\{e,v\}$ satisfying
\[
 g_\M(t)=g_{\M\setminus e}(t)+g_{\M/e}(t)+t\,g_{\N_{e,v}}(t).
\]
Using a fixed total order on $E$, we prescribe the proper cyclic flats
of $\N_{e,v}$ and their ranks. The identity follows from the covaluative
formula of Ferroni and Schr\"oter together with recurrences for its
correction polynomials, derived from Ferroni's enumeration of admissible
Delannoy paths for Schubert matroids. Since all three matroids on the
right have fewer elements, the identity supplies the induction step in
the proof of nonnegativity.



\section{Introduction}
\label{sec:introduction}

Speyer \cite{Speyer2009} introduced the $g$-polynomial in his study of
the $K$-theory of the Grassmannian, and Fink and Speyer
\cite{FinkSpeyer2012} developed this invariant for arbitrary matroids.
For a matroid $\M$, write $g_\M(t)\in\mathbb Z[t]$ for its
$g$-polynomial. Speyer \cite[Proposition~3.3]{Speyer2009} proved that
its coefficients are nonnegative when $\M$ is realizable over a field
of characteristic zero, and conjectured that
\[
 g_{\M}(t)\in\mathbb Z_{\ge0}[t]
\]
for every matroid. Nonnegativity would imply his conjectured upper
bounds on the numbers of bounded faces of tropical linear spaces
\cite{Speyer2008}.

For a matroid $\M$ of rank $r$, Fink, Shaw, and Speyer
\cite{FinkShawSpeyer2026} denote the coefficient of $t^r$ in
$g_\M(t)$ by $\omega(\M)$. Its nonnegativity was proved by Berget
and Fink \cite[Theorem~E]{BergetFink2025}; a different proof was given
by Eur, Fink, and Larson \cite[Corollary~1.2]{EurFinkLarson2025}.
The deduction of nonnegativity of all coefficients in
\cite{BergetFink2025} uses an earlier version of the reduction in
\cite{FinkShawSpeyer2026}. The revised Theorem~1.5 of the latter
paper also requires $\M/A$ to be connected for every proper flat $A$
of $\M$. Thus nonnegativity of $\omega$ alone does not settle the
conjecture. Nonnegativity of all coefficients is known for matroids of
rank at most three \cite[Corollary~5.39]{Larson2026}, and hence also
for matroids of corank at most three by duality, but remains open in
general.

We prove the conjecture for split matroids. Introduced by Joswig and
Schr\"oter \cite{JoswigSchroter2017}, this class is closed under taking
minors and duality and contains all paving and copaving matroids.
A matroid is \emph{paving} if every circuit has cardinality at least
the rank of the matroid, and \emph{copaving} if its dual is paving. It is
\emph{sparse paving} if both it and its dual are paving.
B\'erczi et al.\ \cite[Theorem~11]{BercziEtAl2023} introduced
elementary split matroids and characterized connected split matroids
on at least two elements as those whose proper cyclic flats are
pairwise incomparable. For connected matroids on at least two
elements, the notions of split and elementary split therefore
coincide. This description is the basis of our construction.

Ferroni and Schr\"oter \cite[Theorem~9.13]{FerroniSchroter2024}
obtained a formula for covaluative invariants of connected split
matroids. For a connected split matroid $\M$ of rank $k$ on $n\ge2$
elements, it expresses $g_\M(t)$ as $g_{\U_{k,n}}(t)$ minus a sum
over the proper cyclic flats of $\M$, where $\U_{k,n}$ is the uniform
matroid of rank $k$ on $n$ elements. Each summand depends only on
$k$, $n$, and the rank and cardinality of the corresponding cyclic
flat. They also proved Speyer's conjecture for sparse paving matroids
\cite[Theorem~9.21]{FerroniSchroter2024}. Our main result is the
following theorem.

\begin{theorem}\label{thm:main}
If $\M$ is a split matroid, then
\[
 g_\M(t)\in\mathbb Z_{\ge0}[t].
\]
\end{theorem}

The proof of Theorem~\ref{thm:main} and the auxiliary matroid construction
were obtained independently of Wang's work \cite{Wang2026}, which also
establishes the paving case of Theorem~\ref{thm:main}. By duality,
Wang's result also gives the copaving case.

To prove nonnegativity, we turn the subtraction in the covaluative
formula into a deletion--contraction identity whose terms are
$g$-polynomials of smaller split matroids. For an element $e$ of a
matroid $\M$, write $\M\setminus e$ and $\M/e$ for its deletion and
contraction. The following theorem gives the identity.

\begin{theorem}\label{thm:split-recurrence}
Let $\M$ be a connected split matroid of rank $k$ on an $n$-element
ground set $E$, where $n\ge4$. Suppose that $e\in E$ is such that
$\M\setminus e$ and $\M/e$ are both connected. For every
$v\in E\setminus\{e\}$, there exists a connected elementary split
matroid $\N_{e,v}$ of rank $k-1$ on $E\setminus\{e,v\}$ satisfying
\begin{equation}\label{eq:intro-splitting}
 g_\M(t)=g_{\M\setminus e}(t)+g_{\M/e}(t)+t\,g_{\N_{e,v}}(t).
\end{equation}
\end{theorem}

For a simple, cosimple, connected split matroid, such an element
always exists. In fact, Proposition~\ref{prop:connected-minors} shows
that among the $2n$ deletions and contractions $\M\setminus e$ and
$\M/e$, with $e\in E$, at most one is disconnected. We construct $\N_{e,v}$ in
Section~\ref{sec:residual} by specifying its cyclic flats and their
ranks. The construction uses an arbitrary total order $\prec$ on $E$;
we write $\N_{e,v,\prec}$ when this dependence is relevant.
For fixed $\M$ and $e$, its $g$-polynomial is independent of both
$v$ and $\prec$.

The recurrences for the correction polynomials distinguish whether a
proper cyclic flat contains $e$. They follow from Ferroni's
enumeration of admissible Delannoy paths \cite{Ferroni2023} and
prescribe the ranks and cardinalities of the proper cyclic flats of
$\N_{e,v}$. We realize these data by modifying the cyclic flats of
$\M$. A common total order coordinates the elements removed from or
inserted into these sets and ensures the required intersection bounds.
After proving that the resulting matroid is connected, we combine
the correction recurrences with the corresponding identity for uniform
matroids to obtain \eqref{eq:intro-splitting}.

All three matroids on the right of \eqref{eq:intro-splitting} are
connected split matroids with fewer elements than $\M$. We use
matroids of rank or corank at most two as the initial cases for an
induction on the size of the ground set. Simplification and
cosimplification handle connected matroids that are not simple or
cosimple, while multiplicativity under direct sums handles
disconnected matroids. This proof does not use the $\omega$-invariant
or the reduction in \cite{FinkShawSpeyer2026}.

When $\M$ is paving of rank at least three and both minors at $e$
are connected, \eqref{eq:intro-splitting} agrees with the recurrence
obtained independently by Wang \cite[Proposition~2.5]{Wang2026}. Both arguments use the formula
of Ferroni and Schr\"oter. Wang \cite[Lemma~2.6]{Wang2026} proves
nonnegativity of the remaining polynomial by an inequality for
coefficients. Our construction identifies that polynomial with
$g_{\N_{e,v}}(t)$ and applies also when the proper cyclic flats have
different ranks. Section~\ref{subsec:paving} gives the specialization
and shows that $\N_{e,v}$ is paving when $\M$ is paving.

We begin in Section~\ref{sec:preliminaries} with the cyclic flat
criterion and the properties of the $g$-polynomial needed in the proof.
The recurrences developed in Section~\ref{sec:corrections} determine the
ranks and cardinalities required for the auxiliary matroid.
To realize these data, we first examine deletion and contraction in
Section~\ref{sec:minors}, then give the construction in
Section~\ref{sec:residual}. The two main theorems follow in
Section~\ref{sec:splitting}, which concludes with the specialization
to paving matroids.

\section{Split matroids and Speyer's \texorpdfstring{$g$}{g}-polynomial}
\label{sec:preliminaries}

Throughout this paper, all matroids are finite, and connected matroids
are required to have nonempty ground sets. We follow Oxley
\cite{Oxley2011} for general matroid terminology.

\subsection{Cyclic flats of elementary split matroids}

We follow Ferroni and Schr\"oter
\cite[Sections~2--4]{FerroniSchroter2024} for terminology and notation concerning
cyclic flats and split matroids.
Let $\M$ be a matroid with ground set $E=E(\M)$
and rank function $\rk_{\M}$.
We write $\rk(\M)=\rk_{\M}(E)$ for the rank of $\M$.
A subset $A\subseteq E$ is \emph{cyclic} if it is a union of circuits,
or equivalently, if the restriction $\M|_A$ of $\M$ to $A$ has no
coloops. A \emph{cyclic flat} is a flat that is cyclic. The cyclic
flats of $\M$ form a lattice under inclusion, which we denote by
$\mathcal Z(\M)$. If $\M$ is loopless and coloopless, the least and
greatest elements of $\mathcal Z(\M)$ are $\varnothing$ and $E$,
respectively. In this case, we call the remaining cyclic flats
\emph{proper cyclic flats} and write
$\mathcal Z^{\circ}(\M)=\mathcal Z(\M)\setminus\{\varnothing,E\}$.

We use the following cyclic flat criterion in the form given by Bonin
and de Mier \cite[Theorem~3.2]{BoninDeMier2008}; the notation follows
Ferroni and Schr\"oter \cite[Theorem~2.3]{FerroniSchroter2024}.
Bonin and de Mier also discuss its antecedent in work of Sims at the
end of their Section~3.

\begin{theorem}[Bonin and de Mier {\cite[Theorem~3.2]{BoninDeMier2008}}]
\label{thm:cyclic-flat-axioms}
Let $E$ be a finite set, let $\mathcal Z$ be a collection of subsets of
$E$, and let $\rk:\mathcal Z\to\mathbb Z_{\geq0}$. There exists a matroid
$\M$ on $E$ with $\mathcal Z(\M)=\mathcal Z$ and
$\rk_{\M}(Z)=\rk(Z)$ for every $Z\in\mathcal Z$ if and only if the
following conditions hold:
\begin{itemize}
\item[{\rm (Z0)}] $\mathcal Z$ is a lattice under inclusion.

\item[{\rm (Z1)}] $\rk(0_{\mathcal Z})=0$, where $0_{\mathcal Z}$ denotes
the least element of $\mathcal Z$.

\item[{\rm (Z2)}] If $X,Y\in\mathcal Z$ satisfy $X\subsetneq Y$, then
$0<\rk(Y)-\rk(X)<|Y\setminus X|$.

\item[{\rm (Z3)}] For all $X,Y\in\mathcal Z$,
\[
\rk(X)+\rk(Y)
\geq
\rk(X\vee_{\mathcal Z}Y)
+
\rk(X\wedge_{\mathcal Z}Y)
+
\bigl|(X\cap Y)\setminus(X\wedge_{\mathcal Z}Y)\bigr|.
\]
\end{itemize}
Here $X\vee_{\mathcal Z}Y$ and $X\wedge_{\mathcal Z}Y$ denote the join
and meet of $X$ and $Y$ in the lattice $\mathcal Z$, respectively.
\end{theorem}

B\'erczi et al.\ \cite[Theorem~11]{BercziEtAl2023} proved that a
loopless, coloopless matroid is elementary split if and only if its
proper cyclic flats are pairwise incomparable. For connected matroids
on at least two elements, this is also equivalent to being split.
Distinct proper cyclic flats therefore have meet $\varnothing$ and join
$E$ in the lattice of cyclic flats. Specializing
Theorem~\ref{thm:cyclic-flat-axioms} to this lattice and applying the
description of independent sets due to Bonin and de Mier
\cite[Lemma~3.1(i)]{BoninDeMier2008} gives the following standard
construction. It is also described in B\'erczi et al.\
\cite[Remark~7]{BercziEtAl2023}. We include the verification to fix
the precise inequalities used later.

\begin{proposition}
\label{prop:cyclic-flat-description}
Let $S$ be a finite set and let $k,m$ be integers with
$1\leq k<|S|$ and $m\geq0$. Suppose that $F_1,\ldots,F_m$ are distinct
subsets of $S$ and that $r_1,\ldots,r_m$ are integers satisfying
\begin{align}
1\leq r_i&\leq k-1,
&& \text{for }1\leq i\leq m, \label{eq:rank-condition}\\
|F_i|&\geq r_i+1,
&& \text{for }1\leq i\leq m, \label{eq:size-condition}\\
|S\setminus F_i|&\geq k-r_i+1,
&& \text{for }1\leq i\leq m, \label{eq:complement-condition}\\
|F_i\cap F_j|&\leq r_i+r_j-k,
&& \text{for }1\leq i<j\leq m. \label{eq:intersection-condition}
\end{align}
Then there is a loopless, coloopless elementary split matroid $\M$ of
rank $k$ on $S$ whose proper cyclic flats are exactly $F_1,\ldots,F_m$,
with $\rk_{\M}(F_i)=r_i$ for each $i$. Its independent sets are
\begin{equation}
\label{eq:independent-sets}
\left\{
X\subseteq S:
|X|\leq k
\text{ and }
|X\cap F_i|\leq r_i
\text{ for }1\leq i\leq m
\right\}.
\end{equation}
When $m=0$, this matroid is $\U_{k,|S|}$.
\end{proposition}

\begin{proof}
We verify the axioms of Theorem~\ref{thm:cyclic-flat-axioms} for
$\mathcal Z=\{\varnothing,F_1,\ldots,F_m,S\}$.
Conditions \eqref{eq:rank-condition}--\eqref{eq:complement-condition}
ensure that each $F_i$ is a nonempty proper subset of $S$. For $i\ne j$,
the intersection and rank bounds give
\[
|F_i\cap F_j|
\leq r_i+r_j-k
\leq\min\{r_i,r_j\}-1
<\min\{|F_i|,|F_j|\}.
\]
Thus the $F_i$ are pairwise incomparable. It follows that $\mathcal Z$
is a lattice with least element $\varnothing$ and greatest element $S$.
For $i\ne j$, we have $F_i\wedge F_j=\varnothing$ and $F_i\vee F_j=S$.
This proves {\rm (Z0)}.
Define $\rk(\varnothing)=0$, $\rk(S)=k$, and $\rk(F_i)=r_i$.
Then {\rm (Z1)} holds by definition. The inequalities
$0<r_i<|F_i|$, $0<k-r_i<|S\setminus F_i|$, and $0<k<|S|$
verify {\rm (Z2)} for all strict inclusions in $\mathcal Z$.
For {\rm (Z3)}, comparable pairs give equality, while for distinct
$F_i,F_j$ the required inequality is
$r_i+r_j\ge k+|F_i\cap F_j|$, which is
\eqref{eq:intersection-condition}.

Theorem~\ref{thm:cyclic-flat-axioms} now gives a matroid $\M$ of rank
$k$ on $S$ with the prescribed cyclic flats and ranks. Since its least
and greatest cyclic flats are $\varnothing$ and $S$, it is loopless and
coloopless. Since its proper cyclic flats are pairwise incomparable,
it is elementary split.
By Bonin and de Mier \cite[Lemma~3.1(i)]{BoninDeMier2008}, a set
$X\subseteq S$ is independent in $\M$ if and only if
$|X\cap Z|\leq\rk_{\M}(Z)$ for every $Z\in\mathcal Z$.
The condition for $Z=\varnothing$ is automatic. Those for $Z=S$ and
$Z=F_i$ are precisely the inequalities in \eqref{eq:independent-sets}.
When $m=0$, only the condition $|X|\leq k$ remains, so
$\M=\U_{k,|S|}$.
\end{proof}

We introduce notation for the ranks and cardinalities of cyclic flats
and record the inequalities needed for the constructions below.
Let $\M$ be a loopless, coloopless elementary split matroid of rank $k$
on an $n$-element set $E$. For $F\in\mathcal Z^\circ(\M)$, write
\[
 r_F=\rk_{\M}(F),\qquad h_F=|F|,\qquad
 p_F=k-r_F,\qquad b_F=h_F-r_F,\qquad a_F=n-k-h_F+r_F.
\]
Here $b_F$ is the nullity of $F$, $p_F$ is the rank of the contraction
$\M/F$, and $a_F$ is the rank of $E\setminus F$ in the dual matroid
$\M^*$. Axiom {\rm (Z2)} gives
\begin{equation}\label{eq:positive-parameters}
 r_F,p_F,b_F,a_F\ge1.
\end{equation}
For distinct $F,G\in\mathcal Z^\circ(\M)$, axiom {\rm (Z3)} gives
\begin{equation}\label{eq:H1}
 |F\cap G|\le r_F+r_G-k.
\end{equation}
If $\M$ is simple and cosimple, then
\begin{equation}\label{eq:simple-cosimple-bounds}
 r_F\ge2,\qquad a_F\ge2.
\end{equation}
Indeed, if $r_F=1$, then $F$ contains a parallel pair, contrary to
simplicity. Since $E\setminus F$ is a cyclic flat of $\M^*$ of rank
$a_F$, the same argument applied to $\M^*$ gives $a_F\ge2$.

\subsection{Speyer's \texorpdfstring{$g$}{g}-polynomial}

Following Ferroni \cite[Definition~1.2]{Ferroni2023}, we define
Speyer's $g$-polynomial by the following characterization; see also
Ferroni and Schr\"oter \cite[Theorem~9.15]{FerroniSchroter2024}.
Here a connected \emph{series-parallel matroid} on
at least two elements is a matroid obtained from $\U_{1,2}$ by
successive series and parallel extensions. A parallel extension adds
a new element parallel to an existing nonloop element; a series
extension is the dual operation.

\Needspace{8\baselineskip}
\begin{definition}
Speyer's $g$-polynomial is
the unique matroid invariant $\M\mapsto g_\M(t)\in\mathbb Z[t]$
satisfying the following properties:
\begin{enumerate}
\renewcommand{\labelenumi}{(\roman{enumi})}
\item If $\M$ has a loop or a coloop, then $g_\M(t)=0$.
\item If $\M$ is a connected series-parallel matroid on at least two
elements, then $g_\M(t)=t$.
\item For a direct sum $\M=\M_1\oplus\M_2$, we have
$g_\M(t)=g_{\M_1}(t)g_{\M_2}(t)$.
\item The map $\M\mapsto g_\M(t)$ is covaluative with respect to matroid
polytope subdivisions.
\end{enumerate}
\end{definition}

The \emph{base polytope} $P(\M)$ is the convex hull
in $\mathbb R^{E(\M)}$ of the indicator vectors of the bases of $\M$.
For a subdivision $\mathcal S$ of the base polytope $P(\M)$ of $\M$
into matroid base polytopes, let $\mathcal S^{\circ}$ denote the
collection of cells not contained in the relative boundary of $P(\M)$,
as in Ferroni \cite[Section~2.3]{Ferroni2023}. Property (iv) means that
\[
 g_\M(t)=
 \sum_{P(\N)\in\mathcal S^{\circ}}g_\N(t).
\]
The sum runs over all cells in $\mathcal S^{\circ}$, including those of
lower dimension.
For the general theory of valuative invariants of matroids and
polymatroids, we refer to Derksen and Fink \cite{DerksenFink2010}.
Ardila and Sanchez \cite{ArdilaSanchez2023} develop a framework for
constructing valuative invariants using the Hopf monoid of generalized
permutahedra.

We recall the $2$-sum operation, following Oxley
\cite{Oxley2011}. Let $\M_1$ and $\M_2$ have ground sets $E_1$ and
$E_2$ with $E_1\cap E_2=\{p\}$ and $|E_i|\ge3$ for $i=1,2$,
where $p$ is neither a loop nor a coloop in either matroid.
Their \emph{$2$-sum} $\M_1\oplus_2\M_2$ is the matroid on
$(E_1\cup E_2)\setminus\{p\}$ whose circuits are the circuits of
either $\M_i$ that avoid $p$, together with the sets
$(C_1\cup C_2)\setminus\{p\}$, where $C_i$ is a circuit of
$\M_i$ containing $p$ for $i=1,2$.
We recall three standard properties of the $g$-polynomial.
The $2$-sum formula and duality were established in the work of Speyer
\cite{Speyer2009} and extended to arbitrary matroids by Fink and Speyer
\cite[Section~9]{FinkSpeyer2012}. The $2$-sum formula is also recorded
by Panzer \cite[Section~4.7]{Panzer2025}, and duality by Ferroni and
Schr\"oter \cite[Proposition~9.17(b)]{FerroniSchroter2024}.
For the constant term and the degree bound in part~(iii), see also
Larson \cite[Definition~5.19 and Corollary~5.22]{Larson2026}.
The normalization on the empty matroid follows from multiplicativity
and $g_{\U_{1,2}}(t)=t$.

\begin{proposition}\label{prop:g-properties}
Let $\M$, $\M_1$, and $\M_2$ be matroids.
\begin{enumerate}
\renewcommand{\labelenumi}{(\roman{enumi})}
\item If the $2$-sum $\M_1\oplus_2\M_2$ is defined, then
\[
 g_{\M_1\oplus_2 \M_2}(t)
 =\frac{g_{\M_1}(t)g_{\M_2}(t)}{t}.
\]
\item The polynomial is invariant under duality. That is,
$g_{\M^*}(t)=g_\M(t)$.
\item The empty matroid has $g$-polynomial $1$. If $\M$ is nonempty,
then $g_\M(0)=0$ and $\deg g_\M\le\min\{k,n-k\}$, where
$k=\rk(\M)$ and $n=|E(\M)|$.
The quantity $n-k$ is the \emph{corank} of $\M$.
\end{enumerate}
\end{proposition}

\section{Cuspidal matroids and recurrences}
\label{sec:corrections}

\subsection{The formula of Ferroni and Schr\"oter}
Following Ferroni and Schr\"oter
\cite[Definitions~3.1 and~3.7]{FerroniSchroter2024}, a subset $A$ of
the ground set of a matroid $\M$ is \emph{stressed} if both $\M|_A$
and $\M/A$ are uniform matroids. If $\M$ has rank $k$ and ground set $E$, the
\emph{cover} of $A$ is
\[
 \operatorname{cover}_{\M}(A)
 =\{B\subseteq E:|B|=k,\ |B\cap A|\ge\rk_{\M}(A)+1\}.
\]
Every member of this family is a nonbasis of $\M$.
Let $r,k,h,n$ be integers satisfying
\begin{equation}\label{eq:admissible-cuspidal}
 r\ge1,\qquad p:=k-r\ge1,\qquad
 b:=h-r\ge1,\qquad a:=n-k-h+r\ge1.
\end{equation}
We denote by $\La_{r,k,h,n}$ the cuspidal matroid of rank $k$ on $n$
elements introduced by Ferroni and Schr\"oter
\cite[Definition~3.21]{FerroniSchroter2024}, retaining their order of
parameters. Lattice path matroids were introduced by Bonin, de Mier,
and Noy \cite{BoninDeMierNoy2003}. To describe this matroid by a lattice path,
write $\Nstep=(0,1)$ and $\Estep=(1,0)$ for the north and east steps,
respectively. We write paths as words in $\Nstep$ and $\Estep$, with
powers denoting repeated steps. Following Panzer
\cite[Definition~3.8]{Panzer2025}, for a path $P$ from $(0,0)$ to
$(n-k,k)$, we write $\LM(P)$ for the lattice path matroid with upper
path $P$ and lower path $\Estep^{n-k}\Nstep^k$. Its ground set is
$\{1,\ldots,n\}$, and its bases are the sets of positions of the north
steps in paths that remain weakly above the lower path and weakly below
$P$. Matroids admitting such a presentation, with lower path
$\Estep^{n-k}\Nstep^k$, are called \emph{Schubert matroids}.
Since $p+r=k$ and $a+b=n-k$, the path
$\Nstep^p\Estep^a\Nstep^r\Estep^b$ runs from $(0,0)$ to $(n-k,k)$.
By Ferroni and Schr\"oter
\cite[Proposition~3.25]{FerroniSchroter2024}, we have
\begin{equation}\label{eq:cuspidal-path}
 \La_{r,k,h,n}
 \cong \LM(\Nstep^p\Estep^a\Nstep^r\Estep^b).
\end{equation}

To abbreviate the summand associated with a proper cyclic flat of rank
$r$ and cardinality $h$ in the formula below, we write
\begin{equation}\label{eq:correction-definition}
 \cC_{r,k,h,n}(t)
 :=g_{\La_{r,k,h,n}}(t)
   +g_{\U_{r,h}}(t)g_{\U_{k-r,n-h}}(t).
\end{equation}
The following proposition is the specialization of the covaluative
formula of Ferroni and Schr\"oter
\cite[Theorem~9.13]{FerroniSchroter2024} to Speyer's $g$-polynomial,
with the sum indexed by the proper cyclic flats.

\begin{proposition}\label{prop:FS-split}
Let $\M$ be a loopless, coloopless, connected split matroid of rank $k$
on an $n$-element ground set $E$. Then
\begin{equation}\label{eq:FS-split}
 g_\M(t)=g_{\U_{k,n}}(t)
 -\sum_{F\in\mathcal Z^{\circ}(\M)}
   \cC_{r_F,k,h_F,n}(t).
\end{equation}
Here $r_F=\rk_{\M}(F)$ and $h_F=|F|$.
\end{proposition}

\begin{proof}
Since $\M$ is connected and split, it is elementary split, so its proper
cyclic flats are pairwise incomparable. Hence each proper cyclic flat
$F$ determines a saturated chain
$\varnothing\subsetneq F\subsetneq E$ in $\mathcal Z(\M)$.
Ferroni and Schr\"oter \cite[Proposition~3.9]{FerroniSchroter2024}
identify these subsets with the stressed subsets with nonempty cover.
Their covaluative formula is therefore indexed by the proper cyclic
flats of $\M$. By the multiplicativity of the $g$-polynomial under
direct sums, the summand corresponding to $F$ is
$\cC_{r_F,k,h_F,n}(t)$, as defined in
\eqref{eq:correction-definition}. This proves \eqref{eq:FS-split}.
\end{proof}

\subsection{Uniform matroids}
The following identity is the uniform case of the lattice path
recurrences in Panzer \cite[Section~3.2]{Panzer2025}; it is also stated
explicitly by Wang \cite[Lemma~2.4, equation~(2.2)]{Wang2026}.
We give a coefficient proof using Speyer's formula and Pascal's identity.

\begin{lemma}\label{lem:uniform-recurrence}
If $2\le k\le n-2$, then
\begin{equation}\label{eq:uniform-recurrence}
 g_{\U_{k,n}}(t)
 =g_{\U_{k,n-1}}(t)+g_{\U_{k-1,n-1}}(t)+t g_{\U_{k-1,n-2}}(t).
\end{equation}
\end{lemma}

\begin{proof}
By Speyer \cite[Proposition~10.1]{Speyer2009}, for $1\le k\le n-1$,
\begin{equation}\label{eq:intro-uniform}
 g_{\U_{k,n}}(t)
 =\sum_{i=1}^{\min\{k,n-k\}}
 \binom{n-i-1}{k-i}\binom{n-k-1}{i-1}t^i.
\end{equation}
We take $\binom{m}{j}=0$ unless $0\le j\le m$ for integers $m$ and $j$.
For $i\ge0$, let $u_i(n,k)$ denote the coefficient of $t^i$ in
$g_{\U_{k,n}}(t)$, given by \eqref{eq:intro-uniform}.
In particular, $u_i(n,k)=0$ unless $1\le i\le\min\{k,n-k\}$.

Both sides of \eqref{eq:uniform-recurrence} have zero constant term and
degree at most $\min\{k,n-k\}$. Fix $1\le i\le\min\{k,n-k\}$.
The coefficient of $t^i$ on the right-hand side is
$u_i(n-1,k)+u_i(n-1,k-1)+u_{i-1}(n-2,k-1)$.
By \eqref{eq:intro-uniform} and Pascal's identity,
\[
 u_i(n-1,k)+u_{i-1}(n-2,k-1)
 =
 \binom{n-i-2}{k-i}\binom{n-k-1}{i-1}.
\]
Adding $u_i(n-1,k-1)$ and applying Pascal's identity once more gives
$u_i(n,k)$. The coefficients on the two sides therefore agree.
\end{proof}

\subsection{Recurrences for the polynomials \texorpdfstring{$\cC_{r,k,h,n}(t)$}{C(r,k,h,n;t)}}

The next proposition combines the lattice path recurrences recorded by
Panzer \cite[Section~3.2]{Panzer2025} with
Lemma~\ref{lem:uniform-recurrence} and duality. We express these known
recurrences in terms of the correction polynomials $\cC_{r,k,h,n}$.
The two forms distinguish whether the chosen element $e$ lies in a
proper cyclic flat $F$; the boundary cases account for the disappearance
of that cyclic flat after deletion or contraction.

For a set $A$ and an element $e$, we write
$A-e=A\setminus\{e\}$.
We prove the first recurrence using Ferroni's formula for Schubert
matroids as a weighted sum over admissible Delannoy paths and
Lemma~\ref{lem:uniform-recurrence}. Duality gives the second recurrence.
In polynomial identities, we occasionally omit the
argument $t$ from $g_\M(t)$ and $\cC_{r,k,h,n}(t)$.

\begin{proposition}\label{lem:correction-recurrences}
Let $r,k,h,n$ satisfy \eqref{eq:admissible-cuspidal}, with
$p=k-r$, $b=h-r$, and $a=n-k-h+r$.
\begin{enumerate}
\renewcommand{\labelenumi}{(\roman{enumi})}
\item If $r\ge2$ and $b\ge2$, then
\begin{equation}\label{eq:internal-recurrence}
 \cC_{r,k,h,n}
 =\cC_{r,k,h-1,n-1}
  +\cC_{r-1,k-1,h-1,n-1}
  +t\cC_{r-1,k-1,h-2,n-2}.
\end{equation}
If $r\ge2$ and $b=1$, then
\begin{equation}\label{eq:internal-boundary}
 \cC_{r,k,r+1,n}=\cC_{r-1,k-1,r,n-1}.
\end{equation}

\item If $a\ge2$ and $p\ge2$, then
\begin{equation}\label{eq:external-recurrence}
 \cC_{r,k,h,n}
 =\cC_{r,k,h,n-1}
  +\cC_{r,k-1,h,n-1}
  +t\cC_{r,k-1,h,n-2}.
\end{equation}
If $a\ge2$ and $p=1$, then
\begin{equation}\label{eq:external-boundary}
 \cC_{k-1,k,h,n}=\cC_{k-1,k,h,n-1}.
\end{equation}
\end{enumerate}
\end{proposition}

\begin{proof}
We first prove (i) by treating the two summands in
\eqref{eq:correction-definition} separately. Assume $r\ge2$ and write
$Q=\Nstep^p\Estep^a\Nstep^{r-1}$.
Each upper path used below begins with a north step and ends with an east
step, so the corresponding lattice path matroid is loopless and coloopless.
For such a path $P$, let $\mathcal A(P)$ denote the set of admissible
Delannoy paths defined by Ferroni \cite[Definition~3.1]{Ferroni2023}.
These paths start at $(1,1)$, end at the endpoint of $P$, and use north
steps $\Nstep$, east steps $\Estep$, and diagonal steps $(1,1)$.
They remain weakly below $P$ and have no north step overlapping $P$.
A diagonal step is allowed only if the north step from the same initial
point remains weakly below $P$ and does not overlap $P$.
For $\pi\in\mathcal A(P)$, let $d(\pi)$ denote its number of diagonal steps.
Ferroni \cite[Theorem~3.4]{Ferroni2023} proves that
\[
 g_{\LM(P)}(t)=\sum_{\pi\in\mathcal A(P)}t^{d(\pi)+1}.
\]

Suppose first that $b\ge2$. The upper path $Q\Nstep\Estep^b$ ends at
$(a+b,k)$. We partition $\mathcal A(Q\Nstep\Estep^b)$ according to
the last step of each path. Deleting that step gives an admissible path for
$Q\Estep^b$, $Q\Nstep\Estep^{b-1}$, or $Q\Estep^{b-1}$, according as
the deleted step is north, east, or diagonal. The corresponding penultimate
points are $(a+b,k-1)$, $(a+b-1,k)$, and $(a+b-1,k-1)$,
respectively. In the north and diagonal cases, the truncated path has height at most
$k-1$, so removing the last north step from the upper path does not change
its admissibility. In the east and diagonal cases, removing the last east
step simply restricts the diagram to the width of the truncated path.
Conversely, appending the corresponding step recovers an admissible path.
Indeed, $a+b-1>a$, so the appended north and diagonal steps lie strictly
to the right of the last vertical segment of the upper path.
Deleting and appending the last step are inverse operations.
The north and east cases preserve the weight $t^{d(\pi)+1}$.
In the diagonal case, the original path has one more diagonal step than
the truncated path, so its weight is $t$ times the weight of the
truncated path.

Suppose now that $b=1$. No path in $\mathcal A(Q\Nstep\Estep)$ can
pass through $(a,k)$. An east step into that point would start above
the upper path, a north step would overlap the upper path, and a
diagonal step would start above the upper path, since $k-1>p$.
Thus the last step into $(a+1,k)$ cannot be east.
A final diagonal step would begin at $(a,k-1)$,
where a north step overlaps the upper path, so this is also forbidden.
The last step must therefore be north, and deleting it gives a bijection
with $\mathcal A(Q\Estep)$ that preserves the number of diagonal steps.
The two cases therefore give
\begin{align}
 g_{\LM(Q\Nstep\Estep^b)}
 ={}&g_{\LM(Q\Estep^b)}
   +g_{\LM(Q\Nstep\Estep^{b-1})}
   +t g_{\LM(Q\Estep^{b-1})}
   &&\text{for }b\ge2,\label{eq:path-three-term}\\
 g_{\LM(Q\Nstep\Estep)}
 ={}&g_{\LM(Q\Estep)}.
 &&\label{eq:path-boundary}
\end{align}
These are the two lattice path recurrences from
\cite[Section~3.2]{Panzer2025}, specialized to the paths used here.
Deleting the last north step from the upper path changes
$(r,k,h,n)$ to $(r-1,k-1,h-1,n-1)$, whereas deleting the last east
step changes these parameters to $(r,k,h-1,n-1)$. Deleting both steps
gives $(r-1,k-1,h-2,n-2)$. In each case, $p$ and $a$ remain unchanged.
For $b\ge2$, the lattice path presentation
\eqref{eq:cuspidal-path} therefore turns \eqref{eq:path-three-term}
into the following identity for the $g$-polynomials of cuspidal matroids:
\begin{equation}\label{eq:L-internal}
 g_{\La_{r,k,h,n}}
 =g_{\La_{r,k,h-1,n-1}}
  +g_{\La_{r-1,k-1,h-1,n-1}}
  +t g_{\La_{r-1,k-1,h-2,n-2}}.
\end{equation}
Lemma~\ref{lem:uniform-recurrence} also applies to the
$g$-polynomial of $\U_{r,h}$, since $r\ge2$ and $h-r=b\ge2$.
Multiplying this recurrence by
$g_{\U_{p,n-h}}$ gives a recurrence for the product
$g_{\U_{r,h}}g_{\U_{p,n-h}}$. Since both $p=k-r$ and $n-h$ are
unchanged by the three parameter substitutions above, each product in
this recurrence is the second summand of the corresponding polynomial
$\cC$ in \eqref{eq:internal-recurrence}. Adding the recurrence for these
products to \eqref{eq:L-internal} proves \eqref{eq:internal-recurrence}.

When $b=1$, equation \eqref{eq:path-boundary} gives
\[
 g_{\La_{r,k,r+1,n}}=g_{\La_{r-1,k-1,r,n-1}}.
\]
Since $r\ge2$, we also have $g_{\U_{r,r+1}}=g_{\U_{r-1,r}}=t$. The factor
$g_{\U_{p,n-r-1}}$ is unchanged by the substitution
$(r,k,r+1,n)\mapsto(r-1,k-1,r,n-1)$. Substituting these equalities into
\eqref{eq:correction-definition} proves \eqref{eq:internal-boundary}.

For (ii), Ferroni and Schr\"oter
\cite[Lemma~3.26]{FerroniSchroter2024} give
\[
 \La_{r,k,h,n}^*\cong\La_{a,n-k,n-h,n}.
\]
Together with
$\U_{r,h}^*\cong\U_{h-r,h}$ and
$\U_{p,n-h}^*\cong\U_{n-h-p,n-h}$, the invariance of the
$g$-polynomial under duality in Proposition~\ref{prop:g-properties} gives
\[
 \cC_{r,k,h,n}=\cC_{a,n-k,n-h,n}.
\]
The substitution $(r,k,h,n)\mapsto(a,n-k,n-h,n)$ changes
$(r,p,b,a)$ to $(a,b,p,r)$. Thus $a\ge2$ and $p\ge2$ are precisely
the hypotheses needed to apply part (i) to
$\cC_{a,n-k,n-h,n}$. Applying \eqref{eq:internal-recurrence} and
then using duality to rewrite each polynomial gives
\eqref{eq:external-recurrence}. If $a\ge2$ and $p=1$, applying
\eqref{eq:internal-boundary} to $\cC_{a,n-k,n-h,n}$ instead gives
\eqref{eq:external-boundary}.
\end{proof}
\section{Cyclic flats and connectedness under deletion and contraction}
\label{sec:minors}

Let $\M$ be a connected elementary split matroid of rank $k$ on an
$n$-element ground set $E$, where $n\ge4$. We first describe the
cyclic flats of the deletion and contraction minors under conditions
on the chosen element. When $\M$ is also simple and cosimple, this
description shows that among all single-element deletions and
contractions, at most one is disconnected.

Suppose that $e\in E$ satisfies the conditions
\begin{equation}\label{eq:local-bounds}
 e\in F\ \Longrightarrow\ r_F\ge2,
 \qquad
 e\notin F\ \Longrightarrow\ a_F\ge2,
 \qquad F\in\mathcal Z^{\circ}(\M).
\end{equation}
These conditions hold whenever $\M\setminus e$ and $\M/e$ are both
connected. Indeed, $\M$ and these two minors then have neither loops
nor coloops, since each has at least two elements. Their ranks give
$k-1=\rk(\M/e)\ge1$ and $k=\rk(\M\setminus e)\le n-2$, so
$2\le k\le n-2$. If $e\in F$ and $r_F=1$, cyclicity gives
$|F|\ge2$. Any $f\in F-e$ is parallel to $e$ and therefore becomes
a loop in $\M/e$, a contradiction. If $e\notin F$ and $a_F=1$,
then $E-F$ is a cyclic flat of $\M^*$ of rank one containing $e$.
The same argument shows that some $f\in E-(F\cup\{e\})$ is a loop in
$\M^*/e=(\M\setminus e)^*$, and hence a coloop in $\M\setminus e$,
again a contradiction. If $\M$ is simple and cosimple, then
$2\le k\le n-2$ and \eqref{eq:local-bounds} hold for every $e\in E$
by \eqref{eq:simple-cosimple-bounds}.

\subsection{Proper cyclic flats under deletion and contraction}
The description of independent sets in
Proposition~\ref{prop:cyclic-flat-description} determines both minors.
The same method is used in B\'erczi et al.\
\cite[proof of Theorem~8]{BercziEtAl2023} to prove that elementary split
matroids are closed under taking minors. We record the explicit
cyclic flats and their ranks under \eqref{eq:local-bounds}, removing
the redundant inequalities and checking the remaining conditions.

\begin{lemma}\label{lem:minor-presentations}
Let $\M$ be a connected elementary split matroid of rank $k$ on an
$n$-element ground set $E$, where $2\le k\le n-2$, and let $e\in E$
satisfy \eqref{eq:local-bounds}. Then $\M\setminus e$ and
$\M/e$ are loopless, coloopless elementary split matroids of ranks $k$
and $k-1$, respectively. Their proper cyclic flats, paired with their
ranks, form the sets
\begin{equation}\label{eq:deletion-presentation}
 \cD_e=\{(F-e,r_F):e\in F,\ b_F\ge2\}
 \cup\{(F,r_F):e\notin F\}
\end{equation}
for $\M\setminus e$, and
\begin{equation}\label{eq:contraction-presentation}
 \cK_e=\{(F-e,r_F-1):e\in F\}
 \cup\{(F,r_F):e\notin F,\ p_F\ge2\}
\end{equation}
for $\M/e$. In both formulas, $F$ ranges over $\mathcal Z^{\circ}(\M)$.
\end{lemma}

\begin{proof}
Since $\M$ is connected and $n\ge4$, $e$ is neither a loop nor a coloop.
Thus $\M\setminus e$ and $\M/e$ have ranks $k$ and $k-1$, respectively.
We first describe their independent sets using
\eqref{eq:independent-sets}. A set $I\subseteq E-e$ is independent
in $\M\setminus e$ if and only if $|I|\le k$ and
$|I\cap(F-e)|\le r_F$ for every proper cyclic flat $F$ of $\M$.
When $e\in F$ and $b_F=1$, the latter inequality is automatic because
$|F-e|=r_F$. Removing these inequalities leaves precisely those indexed
by $\cD_e$. Similarly, $I$ is independent in $\M/e$ if and only if
$I\cup\{e\}$ is independent in $\M$. This gives $|I|\le k-1$,
together with $|I\cap(F-e)|\le r_F-1$ when $e\in F$ and
$|I\cap F|\le r_F$ when $e\notin F$. If $e\notin F$ and $p_F=1$,
then $r_F=k-1$, so the latter inequality follows from $|I|\le k-1$.
Removing these inequalities leaves precisely those indexed by $\cK_e$.

It remains to check that both families satisfy the hypotheses of
Proposition~\ref{prop:cyclic-flat-description}. Write $E'=E-e$ and
let $K=k$ for deletion and $K=k-1$ for contraction. For a pair $(R,s)$
in the corresponding family, we must verify that $s$, $K-s$,
$|R|-s$, and $|E'\setminus R|-(K-s)$ are positive. Their values are
\[
\begin{array}{c|c|c|c|c}
\text{minor and case}&s&K-s&|R|-s&|E'\setminus R|-(K-s)\\ \hline
\M\setminus e,\ e\in F&r_F&p_F&b_F-1&a_F\\
\M\setminus e,\ e\notin F&r_F&p_F&b_F&a_F-1\\
\M/e,\ e\in F&r_F-1&p_F&b_F&a_F\\
\M/e,\ e\notin F&r_F&p_F-1&b_F&a_F
\end{array}
\]
In the first row, $b_F\ge2$ is part of the definition of $\cD_e$;
in the last row, $p_F\ge2$ is part of the definition of $\cK_e$.
The bounds $r_F\ge2$ when $e\in F$ and $a_F\ge2$ when
$e\notin F$ are precisely the conditions in \eqref{eq:local-bounds}.
Together with \eqref{eq:positive-parameters}, these inequalities show
that all entries in the four numerical columns are positive.

We next verify the intersection bound. For deletion, it follows from \eqref{eq:H1},
since removing $e$ cannot enlarge an intersection and the assigned
ranks are unchanged. For contraction, write $\epsilon_F=1$ if
$e\in F$ and $\epsilon_F=0$ otherwise. The pair in $\cK_e$ arising
from $F$ is $(F-e,r_F-\epsilon_F)$. For distinct proper cyclic
flats $F,G$ contributing to $\cK_e$, we have
\begin{align*}
 |(F-e)\cap(G-e)|
 &=|F\cap G|-\epsilon_F\epsilon_G\\
 &\le r_F+r_G-k-\epsilon_F\epsilon_G\\
 &\le (r_F-\epsilon_F)+(r_G-\epsilon_G)-(k-1).
\end{align*}
The last inequality follows from
$(1-\epsilon_F)(1-\epsilon_G)\ge0$.
The intersection bound also shows that distinct proper cyclic flats
contributing to either family give distinct subsets of $E'$. Indeed,
if two such flats gave the same set $R$, with assigned ranks $s_F$ and
$s_G$, then
\[
 |R|\le s_F+s_G-K\le s_F-1<s_F+1\le |R|,
\]
a contradiction.

Finally, $2\le k\le n-2$ implies $0<K<|E'|$ for both minors.
Proposition~\ref{prop:cyclic-flat-description} therefore gives a
loopless, coloopless elementary split matroid for each of the families
$\cD_e$ and $\cK_e$. The description of independent sets at the
beginning of the proof identifies these matroids with $\M\setminus e$
and $\M/e$, respectively. Hence the two families give exactly their
proper cyclic flats and ranks, as claimed.
\end{proof}

\subsection{Connected deletions and contractions}

By B\'erczi et al.\ \cite[Theorem~11]{BercziEtAl2023}, every loopless,
coloopless, disconnected elementary split matroid is a direct sum
$\U_{r_1,m_1}\oplus\U_{r_2,m_2}$ with $0<r_i<m_i$ for $i=1,2$.
Consequently, it has exactly two proper cyclic flats: the ground sets
of its two connected components. These cyclic flats partition the
ground set, and their ranks sum to the rank of the matroid.
A \emph{$1$-separation} of $\M$ is a partition
$E(\M)=A\sqcup B$ into two nonempty sets such that
$\rk_{\M}(A)+\rk_{\M}(B)=\rk(\M)$. A nonempty matroid is
connected if and only if it has no $1$-separation.

\begin{proposition}\label{prop:connected-minors}
Let $\M$ be a simple, cosimple, connected elementary split matroid on
the ground set $E$. Then
\[
 \bigl|\{e\in E:\M\setminus e\text{ is disconnected}\}\bigr|
 +\bigl|\{e\in E:\M/e\text{ is disconnected}\}\bigr|
 \le1.
\]
In particular, there exists an element $e\in E$ such that
$\M\setminus e$ and $\M/e$ are both connected.
\end{proposition}

\begin{proof}
Write $k=\rk(\M)$. Simplicity and cosimplicity imply
$2\le k\le |E|-2$ and \eqref{eq:local-bounds} for every element, so
Lemma~\ref{lem:minor-presentations} applies throughout the proof.
Suppose first that $\M\setminus e$ is disconnected
for some $e\in E$. By Lemma~\ref{lem:minor-presentations} and the
characterization of disconnected elementary split matroids recalled above,
$\M\setminus e$ is a direct sum of two uniform matroids, each of
positive rank and corank. Let $A$ and $B$ be the ground sets of its
connected components. The deletion formula in
Lemma~\ref{lem:minor-presentations} gives unique proper cyclic flats
$F,G$ of $\M$ with $F-e=A$ and $G-e=B$. Deletion preserves the
assigned ranks of these cyclic flats, and the ranks of the two
components of $\M\setminus e$ sum to $k$. Hence
\begin{equation}\label{eq:FG-ranks}
 r_F+r_G=k.
\end{equation}
Together with \eqref{eq:H1}, this gives $F\cap G=\varnothing$.
If $e\in F\cup G$, then $F\sqcup G=E$, since $A\sqcup B=E-e$.
Both $F$ and $G$ are nonempty, so \eqref{eq:FG-ranks} would make
$(F,G)$ a $1$-separation of $\M$, contrary to connectedness. Thus
\begin{equation}\label{eq:FG-partition-minus-e}
 e\notin F\cup G,\qquad F\sqcup G=E-e.
\end{equation}

We claim that $F$ and $G$ are the only proper cyclic flats of $\M$.
Suppose, to the contrary, that $H$ is another proper cyclic flat.
By Lemma~\ref{lem:minor-presentations}, it can fail to give a proper
cyclic flat of $\M\setminus e$ only if $e\in H$ and $b_H=1$.
These conditions must hold, since distinct proper cyclic flats
contributing to the deletion give distinct sets, and $A$ and $B$
are its only proper cyclic flats. Thus $r_H=|H|-1$. Since $F$ and
$G$ partition $E-e$, equations \eqref{eq:H1} and
\eqref{eq:FG-ranks} give
\begin{align*}
 r_H
 =|H|-1
 &=|(H-e)\cap F|+|(H-e)\cap G|\\
 &\le (r_H+r_F-k)+(r_H+r_G-k)
 =2r_H-k.
\end{align*}
This implies $r_H\ge k$, contradicting the fact that $H$ is a proper
flat of the rank-$k$ matroid $\M$.

Fix any $f\in F$. We show that both $\M\setminus f$ and $\M/f$
are connected. By Lemma~\ref{lem:minor-presentations}, the proper
cyclic flats of $\M\setminus f$ are $G$, of rank $r_G$, and also
$F-f$, of rank $r_F$, if $b_F\ge2$. The proper cyclic flats of
$\M/f$ are $F-f$, of rank $r_F-1$, and $G$, of rank $r_G$;
the latter occurs because $p_G=k-r_G=r_F\ge2$, by simplicity of $\M$.
Both minors are loopless, coloopless elementary split matroids, so a
disconnected one would have two complementary proper cyclic flats.
However, every proper cyclic flat just listed avoids $e$, whereas
$e\in E-f$. Their union therefore cannot be the ground set of either
minor. This proves that $\M\setminus f$ and $\M/f$ are connected.
Interchanging $F$ and $G$ proves the same assertion for every $f\in G$.
Since $F\sqcup G=E-e$, all deletions and contractions at elements
other than $e$ are therefore connected.

We also claim that $\M/e$ is connected. Since
$p_F=r_G\ge2$ and $p_G=r_F\ge2$, its proper cyclic flats are
$F$ and $G$, with ranks $r_F$ and $r_G$, by
Lemma~\ref{lem:minor-presentations}. Their ranks sum to $k$, whereas
$\rk(\M/e)=k-1$. They cannot be the two connected components of a
disconnected elementary split matroid. Thus $\M/e$ is connected,
and $\M\setminus e$ is the only disconnected minor among all the
single-element deletions and contractions of $\M$.

It remains to consider the case in which every deletion of $\M$ is
connected. If every contraction is also connected, the assertion
holds. Otherwise, choose $e\in E$ such that
$\M/e$ is disconnected. The dual $\M^*$ is again a simple, cosimple,
connected elementary split matroid, and
$\M^*\setminus e=(\M/e)^*$ is disconnected. The preceding argument
shows that $\M^*\setminus e$ is the only disconnected single-element
deletion or contraction of $\M^*$. Since, for every $f\in E$,
$(\M^*\setminus f)^*=\M/f$ and
$(\M^*/f)^*=\M\setminus f$, invariance of connectedness under duality
shows that $\M/e$ is the only disconnected single-element deletion
or contraction of $\M$. This proves the inequality. Finally,
$|E|\ge4$, so at least one element has both minors connected.
\end{proof}

\section{An auxiliary split matroid}
\label{sec:residual}

Throughout this section, let $\M$ be a connected elementary split
matroid of rank $k$ on an $n$-element ground set $E$, where $n\ge4$,
and fix $e\in E$ such that $\M\setminus e$ and $\M/e$ are both
connected. As shown in Section~\ref{sec:minors}, we have
$2\le k\le n-2$ and \eqref{eq:local-bounds}. We construct a connected
elementary split matroid $\N_{e,v}$ using the ranks and cardinalities
prescribed by Proposition~\ref{lem:correction-recurrences}.
Its $g$-polynomial will supply the remaining term in the
deletion--contraction identity proved in Section~\ref{sec:splitting}.

\subsection{Construction of the sets \texorpdfstring{$R_F$}{RF}}
Fix $v\in E-e$ and a total order $\prec$ on $E$. We modify the
proper cyclic flats of $\M$, using this order for all choices of
elements to be removed or inserted. Write
\begin{equation}\label{eq:Nv-ground}
 E'=E-\{e,v\},\qquad K=k-1.
\end{equation}
All minima below are taken with respect to this order.
For each $F\in\mathcal Z^{\circ}(\M)$ with $e\in F$ and $b_F\ge2$,
write $s_F=r_F-1$, let $x_F=\min(F-e)$, and define
\begin{equation}\label{eq:Nv-I}
 R_F=
 \begin{cases}
  F-\{e,v\},&v\in F,\\
  F-\{e,x_F\},&v\notin F.
 \end{cases}
\end{equation}
Since $|F-e|\ge2$, the element $x_F$ is well defined. In either case,
$R_F\subseteq E'$ and $|R_F|=h_F-2$; we assign rank $s_F$ to $R_F$.

For each $F\in\mathcal Z^{\circ}(\M)$ with $e\notin F$ and $p_F\ge2$,
write $s_F=r_F$, let $y_F=\min\bigl(E-(F\cup\{e\})\bigr)$, and define
\begin{equation}\label{eq:Nv-O}
 R_F=
 \begin{cases}
  F,&v\notin F,\\
  (F-v)\cup\{y_F\},&v\in F.
 \end{cases}
\end{equation}
Since $|E-(F\cup\{e\})|=p_F+a_F-1\ge3$, the element $y_F$ is
well defined. If $v\in F$, then $y_F\notin F\cup\{e\}$, so replacing
$v$ by $y_F$ preserves the cardinality of $F$ and produces a subset of
$E'$. Thus $R_F\subseteq E'$ and $|R_F|=h_F$ in both cases; we assign
rank $s_F$ to $R_F$.
Let $\cR_{e,v}$ denote the family of pairs $(R_F,s_F)$ defined in
\eqref{eq:Nv-I} and \eqref{eq:Nv-O}.

\begin{example}\label{ex:ordered-construction}
Let $E=\{1,\ldots,10\}$ with its usual order, and consider the
matroid $\M$ of rank $5$ with proper cyclic flats
\[
 F=\{1,2,3,4,5,6\},\quad r_F=4,
 \qquad G=\{5,6,7,8\},\quad r_G=3.
\]
The hypotheses of Proposition~\ref{prop:cyclic-flat-description} hold;
in particular, $|F\cap G|=2=r_F+r_G-5$. Thus these subsets and ranks
define a loopless, coloopless elementary split matroid $\M$.
The ranks of its proper cyclic flats are $4$ and $3$, and the ranks
of their complements in $\M^*$ are $a_F=3$ and $a_G=4$.
Hence $\M$ is simple and cosimple. Moreover, $F$ and $G$ do not
partition $E$, so $\M$ is connected by the characterization of
disconnected elementary split matroids used in
Section~\ref{sec:minors}.
The matroid is neither paving nor copaving. Indeed, $G$ is cyclic
of rank $3$ and cardinality $4$, so it is a four-element circuit of
the rank-$5$ matroid $\M$. Similarly, $E-F$ is a cyclic flat of
$\M^*$ of rank $3$ and cardinality $4$, hence a four-element circuit
of the rank-$5$ matroid $\M^*$.

Choose $e=1$ and $v=7$. Lemma~\ref{lem:minor-presentations} shows
that the proper cyclic flats of both $\M\setminus1$ and $\M/1$
are $F-1$ and $G$. Their ranks are $4$ and $3$ in the deletion,
and $3$ and $3$ in the contraction. Since neither set contains $9$
or $10$, they do not partition $E-1$, so both minors are connected.
The conditions $e\in F$, $b_F=2$, $e\notin G$, and $p_G=2$ ensure
that both $F$ and $G$ contribute to $\cR_{1,7}$.
Using the usual order on $E$, we obtain $x_F=y_G=2$ and
\[
 R_F=\{3,4,5,6\},\quad s_F=3,
 \qquad R_G=\{2,5,6,8\},\quad s_G=3.
\]
Both sets have cardinality $4$ and assigned rank $3$. Since the new
ground set $E-\{1,7\}$ has cardinality $8$ and $K=4$, the rank, size,
and complement conditions of Proposition~\ref{prop:cyclic-flat-description}
hold. The intersection condition holds with equality:
$|R_F\cap R_G|=2=s_F+s_G-K$.

This example also explains why the choices of elements must be
coordinated. If we kept $x_F=2$ but replaced $v=7$ in $G$ by $3$,
the second set would be $R'_G=\{3,5,6,8\}$, still with cardinality
$4$ and assigned rank $s_G=3$. However,
\[
 |R_F\cap R'_G|=3>2=s_F+s_G-K,
\]
so the intersection condition would fail. No common total order can
give both $x_F=2$ and $y_G=3$: the first choice requires $2$ to precede
$3$, whereas the second requires $3$ to precede $2$.
The proof of Proposition~\ref{prop:N-matroid} uses the same observation
to establish the intersection bound when $e\in F$ and $e\notin G$.
\end{example}

\subsection{Existence and connectedness of \texorpdfstring{$\N_{e,v}$}{N}}

\begin{proposition}\label{prop:N-matroid}
There exists a connected elementary split matroid $\N_{e,v}$ of rank
$k-1$ on $E'=E-\{e,v\}$ whose proper cyclic flats are exactly the
sets $R_F$ occurring in $\cR_{e,v}$, with
$\rk_{\N_{e,v}}(R_F)=s_F$ for every $(R_F,s_F)\in\cR_{e,v}$.
\end{proposition}

\begin{proof}
We first show that the family $\cR_{e,v}$ satisfies the hypotheses of
Proposition~\ref{prop:cyclic-flat-description}. Recall that
$E'=E-\{e,v\}$ and $K=k-1$. For each pair $(R_F,s_F)$, the
construction gives
\[
\begin{array}{c|ccccc}
\text{condition}&s_F&|R_F|&K-s_F&|R_F|-s_F
 &|E'\setminus R_F|-(K-s_F)\\ \hline
e\in F&r_F-1&h_F-2&p_F&b_F-1&a_F\\
e\notin F&r_F&h_F&p_F-1&b_F&a_F-1
\end{array}
\]
When $e\in F$, \eqref{eq:local-bounds} gives $r_F\ge2$, and the
construction requires $b_F\ge2$. When $e\notin F$, the construction
requires $p_F\ge2$, and \eqref{eq:local-bounds} gives $a_F\ge2$. Together with
\eqref{eq:positive-parameters}, these bounds show that every pair
satisfies $1\le s_F\le K-1$, $|R_F|\ge s_F+1$, and
$|E'\setminus R_F|\ge K-s_F+1$.

We next verify the intersection condition
\begin{equation}\label{eq:Nv-H1}
 |R_F\cap R_G|\le s_F+s_G-K
\end{equation}
for pairs in $\cR_{e,v}$ arising from distinct proper cyclic flats
$F,G$ of $\M$. If $e\in F\cap G$, then
$R_F\subseteq F-e$ and $R_G\subseteq G-e$, so \eqref{eq:H1} gives
\[
 |R_F\cap R_G|
 \le |F\cap G|-1
 \le r_F+r_G-k-1
 =s_F+s_G-K.
\]

Suppose that $e\notin F\cup G$. We claim that
$|R_F\cap R_G|\le|F\cap G|+1$. If $v$ belongs to neither $F$ nor
$G$, both sets are unchanged. If $v$ belongs to exactly one of them,
only that set acquires a new element, so the intersection increases
by at most one. If $v\in F\cap G$, removing $v$ decreases the
intersection by one, and inserting the two replacement elements
increases it by at most two. Thus, in all cases,
\[
 |R_F\cap R_G|
 \le |F\cap G|+1
 \le r_F+r_G-k+1
 =s_F+s_G-K.
\]

By symmetry, the remaining case is $e\in F$ and $e\notin G$.
Since $s_F=r_F-1$ and $s_G=r_G$, equation \eqref{eq:H1} shows that
it suffices to prove $|R_F\cap R_G|\le|F\cap G|$.
If $v\notin G$, then $R_G=G$ and $R_F\subseteq F$,
so $|R_F\cap R_G|\le|F\cap G|$. If $v\in F\cap G$, deleting $v$
reduces the intersection by one, and inserting $y_G$ increases it by
at most one. It remains to consider $v\notin F$ and $v\in G$.
If the intersection increased, we would have $x_F\notin G$ and
$y_G\in F-\{e,x_F\}$. Since $x_F\ne e$, the first condition places
$x_F$ in $E-(G\cup\{e\})$. The definition of $y_G$ as the least
element of this set gives $y_G\preceq x_F$.
The second condition and $x_F=\min(F-e)$ give $x_F\prec y_G$, a
contradiction. This proves \eqref{eq:Nv-H1} in the remaining case
and is the only part of the proof that uses a common order for the
choices of $x_F$ and $y_G$.

The intersection bound also implies that distinct original cyclic
flats give distinct sets $R_F$. Indeed, if $R_F=R_G=R$ for $F\ne G$,
then $|R|\le s_F+s_G-K\le s_F-1$, contradicting
$|R|\ge s_F+1$. Since $0<K=k-1<n-2=|E'|$, all hypotheses of
Proposition~\ref{prop:cyclic-flat-description} now hold. We obtain a
loopless, coloopless elementary split matroid $\N_{e,v}$ on $E'$ whose
proper cyclic flats are exactly the sets $R_F$, with ranks $s_F$.

It remains to prove that $\N_{e,v}$ is connected. If $\cR_{e,v}$ is
empty, then $\N_{e,v}=\U_{k-1,n-2}$, which is connected because its
rank and corank are positive. Suppose that $\cR_{e,v}$ is nonempty.
If $\N_{e,v}$ were disconnected, the characterization of
B\'erczi et al.\ \cite[Theorem~11]{BercziEtAl2023} would give two
uniform connected components. Their ground sets would be complementary
proper cyclic flats $R_F$ and $R_G$, with
\begin{equation}\label{eq:Nv-rank-sum}
 s_F+s_G=K=k-1.
\end{equation}
If $e\notin F\cup G$, then $s_F=r_F$ and $s_G=r_G$, so
$r_F+r_G=k-1$. This contradicts \eqref{eq:H1}, which would give
$|F\cap G|\le-1$. If $e\in F\cap G$, equation \eqref{eq:Nv-rank-sum}
gives $r_F+r_G=k+1$. Since $e\in F\cap G$, \eqref{eq:H1} forces
$F\cap G=\{e\}$. Since $R_F$ and $R_G$ partition $E'$, their
cardinalities give $h_F+h_G=n+2$. Consequently,
\[
 |F\cup G|=h_F+h_G-|F\cap G|=n+1,
\]
contrary to $F\cup G\subseteq E$.

In the remaining case, we may assume $e\in F$ and $e\notin G$. Equation
\eqref{eq:Nv-rank-sum} gives $r_F+r_G=k$, and \eqref{eq:H1} then implies
$F\cap G=\varnothing$. Since $R_F$ and $R_G$ are complementary in $E'$,
$(h_F-2)+h_G=n-2$, so $h_F+h_G=n$. Hence $F$ and $G$ are nonempty
sets that partition $E$ and satisfy $r_F+r_G=k$. They therefore give
a $1$-separation of $\M$, contrary to its connectedness.
Thus $\N_{e,v}$ is connected.
\end{proof}

\Needspace{4\baselineskip}
\subsection{A formula for the \texorpdfstring{$g$}{g}-polynomial of \texorpdfstring{$\N_{e,v}$}{N}}

Let $\mathcal I_e$ and $\mathcal O_e$ denote the two families of
proper cyclic flats of $\M$ that contribute to the construction:
\[
 \mathcal I_e
 =\{F\in\mathcal Z^{\circ}(\M):e\in F,\ b_F\ge2\},
 \qquad
 \mathcal O_e
 =\{F\in\mathcal Z^{\circ}(\M):e\notin F,\ p_F\ge2\}.
\]
To make the dependence of the constructed matroid on the total order
explicit, we write $\N_{e,v,\prec}$ for the matroid obtained using
the element $v$ and the total order $\prec$ on $E$.

\begin{corollary}\label{cor:N-polynomial}
Let $v\in E-e$ and let $\prec$ be a total order on $E$. The matroid
$\N_{e,v,\prec}$ constructed above has $g$-polynomial
\begin{equation}\label{eq:Nv-polynomial}
 g_{\N_{e,v,\prec}}(t)
 =g_{\U_{k-1,n-2}}(t)
 -\sum_{F\in\mathcal I_e}
   \cC_{r_F-1,k-1,h_F-2,n-2}(t)
 -\sum_{F\in\mathcal O_e}
   \cC_{r_F,k-1,h_F,n-2}(t).
\end{equation}
In particular, for fixed $\M$ and $e$, the $g$-polynomial of
$\N_{e,v,\prec}$ is independent of the choice of $v$ and of the
total order $\prec$.
\end{corollary}

\begin{proof}
By Proposition~\ref{prop:N-matroid}, the matroid $\N_{e,v,\prec}$ is
connected and elementary split, with rank $k-1$ and ground set of
cardinality $n-2$. The map $F\mapsto R_F$ is a bijection from
$\mathcal I_e\sqcup\mathcal O_e$ to
$\mathcal Z^{\circ}(\N_{e,v,\prec})$. The construction gives
\[
 \bigl(\rk_{\N_{e,v,\prec}}(R_F),|R_F|\bigr)=
 \begin{cases}
  (r_F-1,h_F-2),&F\in\mathcal I_e,\\
  (r_F,h_F),&F\in\mathcal O_e.
 \end{cases}
\]
Since $k-1$ and $n-k-1$ are positive, this connected matroid has
neither loops nor coloops. We may therefore apply
Proposition~\ref{prop:FS-split} to its $g$-polynomial. Substituting
the displayed ranks and cardinalities and separating the sum according
to $F\in\mathcal I_e$ or $F\in\mathcal O_e$ gives
\eqref{eq:Nv-polynomial}. The bijection ensures that each proper
cyclic flat contributes once, even when distinct cyclic flats have
the same rank and cardinality.

If $\mathcal I_e\sqcup\mathcal O_e$ is empty, then
$\N_{e,v,\prec}=\U_{k-1,n-2}$ and both sums in
\eqref{eq:Nv-polynomial} vanish, so the same formula applies.
For fixed $\M$ and $e$, the indexing sets $\mathcal I_e$ and
$\mathcal O_e$, as well as the parameters $r_F$ and $h_F$, are
independent of $v$ and $\prec$. Thus the right-hand side of
\eqref{eq:Nv-polynomial}, and hence the $g$-polynomial of
$\N_{e,v,\prec}$, is independent of both choices.
\end{proof}

\begin{remark}\label{rem:N-invariants}
The formulas for split matroids give a stronger independence statement.
For fixed $\M$ and $e$, every
isomorphism-invariant valuative or
covaluative invariant has the same value on all the matroids
$\N_{e,v,\prec}$. Indeed, their rank and cardinality are fixed, and
the preceding proof gives the same multiset of ranks and cardinalities
of proper cyclic flats for every choice of $v$ and $\prec$.
The formulas of Ferroni and Schr\"oter
\cite[Theorems~5.3 and~9.13]{FerroniSchroter2024} determine the value
of such an invariant from these data. In particular, the Tutte,
Kazhdan--Lusztig, and $Z$-polynomials are independent of both choices;
see Ardila and Sanchez \cite[Section~8]{ArdilaSanchez2023} for the
valuativeness of the Kazhdan--Lusztig polynomial, and
\cite[Theorem~9.3]{FerroniSchroter2024} for the latter two
polynomials.
\end{remark}
\section{The deletion--contraction identity and nonnegativity}
\label{sec:splitting}

\subsection{A deletion--contraction identity for the \texorpdfstring{$g$}{g}-polynomial}

\begin{proof}[Proof of Theorem~\ref{thm:split-recurrence}]
Let $e$ be the element in the statement, so that $\M\setminus e$
and $\M/e$ are connected. Both minors have at least three elements
and hence have neither loops nor coloops. Their ranks and coranks
therefore give $2\le k\le n-2$.
Since $\M$ is connected and split, it is elementary split. Fix
$v\in E\setminus\{e\}$ and a total order on $E$, and let $\N_{e,v}$
be the matroid of Proposition~\ref{prop:N-matroid}.
That proposition and Lemma~\ref{lem:minor-presentations} show that
the two minors and $\N_{e,v}$ are loopless, coloopless, connected
elementary split matroids.

For each proper cyclic flat $F$ of $\M$, write $r=r_F$ and $h=h_F$.
We express the $g$-polynomials of all three smaller matroids using
sums indexed by $\mathcal Z^{\circ}(\M)$. For this purpose, define
\begin{align}
D_F&=
\begin{cases}
 \cC_{r,k,h-1,n-1},&e\in F,\ b_F\ge2,\\
 0,&e\in F,\ b_F=1,\\
 \cC_{r,k,h,n-1},&e\notin F,
\end{cases}
\label{eq:deletion-correction}\\
K_F&=
\begin{cases}
 \cC_{r-1,k-1,h-1,n-1},&e\in F,\\
 \cC_{r,k-1,h,n-1},&e\notin F,\ p_F\ge2,\\
 0,&e\notin F,\ p_F=1,
\end{cases}
\label{eq:contraction-correction}\\
T_F&=
\begin{cases}
 \cC_{r-1,k-1,h-2,n-2},&e\in F,\ b_F\ge2,\\
 \cC_{r,k-1,h,n-2},&e\notin F,\ p_F\ge2,\\
 0,&\text{otherwise}.
\end{cases}
\label{eq:N-correction}
\end{align}
By Lemma~\ref{lem:minor-presentations}, $D_F$ and $K_F$ are the
contributions of $F$ to the sums in Proposition~\ref{prop:FS-split}
for $\M\setminus e$ and $\M/e$, respectively; the value is zero
when $F$ gives no proper cyclic flat of the corresponding minor.
Corollary~\ref{cor:N-polynomial} gives the same interpretation of
$T_F$ for $\N_{e,v}$. Each proper cyclic flat of a smaller matroid
arises from exactly one proper cyclic flat of $\M$, so
\begin{align*}
 g_{\M\setminus e}&=g_{\U_{k,n-1}}-\sum_F D_F,\\
 g_{\M/e}&=g_{\U_{k-1,n-1}}-\sum_F K_F,\\
 g_{\N_{e,v}}&=g_{\U_{k-1,n-2}}-\sum_F T_F.
\end{align*}
Here and below, each sum runs over $F\in\mathcal Z^{\circ}(\M)$.

We now compare the contributions of each proper cyclic flat $F$.
By \eqref{eq:local-bounds}, we have $r_F\ge2$ when $e\in F$ and
$a_F\ge2$ when $e\notin F$. Thus parts (i) and (ii) of
Proposition~\ref{lem:correction-recurrences} apply in these two
cases, respectively. The boundary identities in that proposition
cover $b_F=1$ and $p_F=1$, respectively. Thus, for every
$F\in\mathcal Z^{\circ}(\M)$,
\begin{equation}\label{eq:C-splitting}
 \cC_{r_F,k,h_F,n}=D_F+K_F+tT_F.
\end{equation}
Since $2\le k\le n-2$, Lemma~\ref{lem:uniform-recurrence} also gives
the required identity for the $g$-polynomials of the uniform matroids.
Combining it with the three expressions above and
\eqref{eq:C-splitting}, we obtain
\[
 g_{\M\setminus e}+g_{\M/e}+t\,g_{\N_{e,v}}
 =g_{\U_{k,n}}-\sum_F\cC_{r_F,k,h_F,n}
 =g_\M,
\]
where the last equality is Proposition~\ref{prop:FS-split} applied
to $\M$. This proves \eqref{eq:intro-splitting}.
\end{proof}

\begin{example}\label{ex:recurrence-computation}
Consider the matroid $\M$ of rank five on ten elements from
Example~\ref{ex:ordered-construction}, with $e=1$ and $v=7$.
Its proper cyclic flats have rank and cardinality pairs $(4,6)$ and
$(3,4)$. The corresponding pairs for $\M\setminus1$ are $(4,5)$
and $(3,4)$, whereas those for $\M/1$ are $(3,5)$ and $(3,4)$.
The two proper cyclic flats of $\N_{1,7}$ both have rank $3$ and
cardinality $4$.

The correction formulas of Section~\ref{sec:corrections} give
\[
 \cC_{4,5,6,10}(t)=\cC_{3,5,4,10}(t)=5t+8t^2+3t^3.
\]
For either minor, the sum of the two correction polynomials is
$5t+7t^2+2t^3$. Each correction polynomial for $\N_{1,7}$ is
$t+t^2$. Applying Proposition~\ref{prop:FS-split} and
\eqref{eq:intro-uniform}, and using duality for the uniform terms
of the two minors, we obtain
\begin{align*}
 g_\M(t)
 &=g_{\U_{5,10}}(t)-2(5t+8t^2+3t^3)\\
 &=60t+124t^2+84t^3+20t^4+t^5,\\
 g_{\M\setminus1}(t)=g_{\M/1}(t)
 &=g_{\U_{4,9}}(t)-(5t+7t^2+2t^3)\\
 &=30t+53t^2+28t^3+4t^4,\\
 g_{\N_{1,7}}(t)
 &=g_{\U_{4,8}}(t)-2(t+t^2)\\
 &=18t+28t^2+12t^3+t^4.
\end{align*}
These expressions give
\[
 g_\M(t)-g_{\M\setminus1}(t)-g_{\M/1}(t)
 =t^2(18+28t+12t^2+t^3)=t g_{\N_{1,7}}(t).
\]
Thus the auxiliary matroid constructed in
Example~\ref{ex:ordered-construction} realizes the remainder in
the deletion--contraction identity for this matroid.
\end{example}

\subsection{Proof of Theorem~\ref{thm:main}}

Following Oxley \cite{Oxley2011}, we write
$\operatorname{si}(\M)$ for the simplification of $\M$ and
$\operatorname{co}(\M)=(\operatorname{si}(\M^*))^*$ for its cosimplification.
The next lemma records standard consequences of connectedness and the
invariance of $g$ under series and parallel extensions, established in
\cite{Speyer2009} and extended to arbitrary matroids in
\cite[Section~9]{FinkSpeyer2012}. Compare also
\cite[Proposition~9.17(c)]{FerroniSchroter2024} for simplification.
We include a proof using the $2$-sum formula, with the rank and corank
restrictions stated explicitly.

\begin{lemma}
\label{lem:simplification}
Let $\M$ be a connected matroid.
\begin{enumerate}
\renewcommand{\labelenumi}{(\roman{enumi})}
\item If $\M$ has rank at least two, then its simplification
$\operatorname{si}(\M)$ is connected, and
$g_{\operatorname{si}(\M)}(t)=g_\M(t)$.
\item If $\M$ has corank at least two, then its cosimplification
$\operatorname{co}(\M)$ is connected, and
$g_{\operatorname{co}(\M)}(t)=g_\M(t)$.
\end{enumerate}
\end{lemma}

\begin{proof}
We first prove (i). Since $\M$ is connected and has rank at least
two, it has no loops. If $\M$ is simple, the assertion is immediate.
Otherwise, choose distinct parallel elements $e,f\in E(\M)$.
Deleting $e$ preserves the rank because $e$ is parallel to $f$.
It also preserves connectedness. To see this, let $x,y$ be distinct
elements of $E(\M)\setminus\{e\}$. Since $\M$ is connected, there
is a circuit $C$ containing $x$ and $y$. If $e\notin C$, then $C$
is a circuit of $\M\setminus e$. If $e\in C$, then $f\notin C$,
since otherwise $\{e,f\}$ would be a proper subset of $C$ that is
itself a circuit. As $e$ and $f$ are parallel, replacing $e$ by $f$
gives a circuit $(C-\{e\})\cup\{f\}$ of $\M\setminus e$ containing
$x$ and $y$. Thus $\M\setminus e$ is connected.

We next show that deleting $e$ preserves the $g$-polynomial.
Since $\M\setminus e$ is connected and has rank at least two, $f$
is neither a loop nor a coloop in $\M\setminus e$. Take a copy of
$\U_{1,3}$ whose ground set meets $E(\M)\setminus\{e\}$ only in
$f$. The $2$-sum of $\M\setminus e$ and this copy along $f$ replaces
$f$ by two parallel elements and is therefore isomorphic to $\M$.
By Proposition~\ref{prop:g-properties} and $g_{\U_{1,3}}(t)=t$,
\[
 g_\M(t)=\frac{g_{\M\setminus e}(t)g_{\U_{1,3}}(t)}{t}
       =g_{\M\setminus e}(t).
\]
Repeating this deletion until one element remains in each parallel
class yields $\operatorname{si}(\M)$. Each deletion preserves the
rank, connectedness, and the $g$-polynomial, proving (i).

For (ii), the dual matroid $\M^*$ is connected and has rank at least
two. By (i), $\operatorname{si}(\M^*)$ is connected and has the same
$g$-polynomial as $\M^*$. Taking duals gives the connected matroid
$\operatorname{co}(\M)=(\operatorname{si}(\M^*))^*$. The invariance
of the $g$-polynomial under duality gives
\[
 g_{\operatorname{co}(\M)}(t)
 =g_{\operatorname{si}(\M^*)}(t)
 =g_{\M^*}(t)
 =g_\M(t),
\]
which proves (ii).
\end{proof}

\begin{remark}
The rank and corank assumptions in Lemma~\ref{lem:simplification}
are necessary. For $n\ge2$, we have $g_{\U_{1,n}}(t)=t$, but the
simplification of $\U_{1,n}$ is $\U_{1,1}$, whose $g$-polynomial is
zero because it has a coloop. Taking duals shows that cosimplification
can likewise change the $g$-polynomial of a matroid of corank one.
\end{remark}

For the initial cases of the induction, we give an elementary proof of
the following known nonnegativity result. It follows already from
Speyer's theorem \cite{Speyer2009}, since matroids of rank or corank
at most two are realizable in characteristic zero; see also the stronger
rank-three result in \cite[Corollary~5.39]{Larson2026}.

\begin{proposition}\label{prop:small-rank}
Let $\M$ be a connected matroid of rank at most two or corank at
most two. Then the $g$-polynomial of $\M$ has nonnegative coefficients.
\end{proposition}

\begin{proof}
Connectedness is preserved under duality, and
Proposition~\ref{prop:g-properties} gives $g_{\M^*}(t)=g_\M(t)$.
Since the rank of $\M^*$ is the corank of $\M$, it suffices to
consider matroids of rank at most two. If $\M$ has a loop or a
coloop, then $g_\M(t)=0$. We may therefore assume that $\M$ is
loopless and coloopless. Its ground set is nonempty by connectedness,
so its rank is positive. If $\M$ has rank one, then
$\M\cong\U_{1,n}$, where $n=|E(\M)|\ge2$. The formula
\eqref{eq:intro-uniform} gives $g_\M(t)=t$.

Suppose now that $\M$ has rank two, and let $m$ be the number of its
parallel classes. Its simplification is the uniform matroid
$\U_{2,m}$. By Lemma~\ref{lem:simplification}, this matroid is
connected and has the same $g$-polynomial as $\M$. In particular,
$m\ge3$, since $\U_{2,2}$ is disconnected. Applying
\eqref{eq:intro-uniform}, we obtain
$g_\M(t)=g_{\U_{2,m}}(t)=(m-2)t+(m-3)t^2$.
Both coefficients are nonnegative because $m\ge3$, completing the proof.
\end{proof}

\begin{proof}[Proof of Theorem~\ref{thm:main}]
The empty matroid has $g$-polynomial $1$, and every matroid with a loop
or a coloop has zero $g$-polynomial. We may therefore restrict attention
to nonempty split matroids without loops or coloops. We first prove the
assertion for connected split matroids by strong induction on the
cardinality of the ground set. Let $\M$ be such a matroid, write
$n=|E(\M)|$ and $k=\rk(\M)$, and assume that the $g$-polynomial
of every connected split matroid on fewer than $n$ elements has
nonnegative coefficients.

If $k\le2$ or $n-k\le2$, the assertion follows from
Proposition~\ref{prop:small-rank}. We may thus assume $k\ge3$ and
$n-k\ge3$. If $\M$ is not simple, its simplification
$\operatorname{si}(\M)$ has fewer than $n$ elements and is split,
since it is a minor of $\M$. Lemma~\ref{lem:simplification} shows
that $\operatorname{si}(\M)$ is connected and that
$g_\M(t)=g_{\operatorname{si}(\M)}(t)$. The induction hypothesis
therefore proves the assertion for $\M$. If $\M$ is not cosimple,
the same argument applies to $\operatorname{co}(\M)$, using the
corank assumption and part (ii) of Lemma~\ref{lem:simplification}.
It remains to consider the case in which $\M$ is simple and cosimple.

Since $\M$ is connected and split, it is elementary split.
Proposition~\ref{prop:connected-minors} gives an element $e$ for which
both $\M\setminus e$ and $\M/e$ are connected. Choose
$v\in E(\M)\setminus\{e\}$. Theorem~\ref{thm:split-recurrence}
then supplies a connected split matroid $\N_{e,v}$ satisfying
\[
 g_\M(t)=g_{\M\setminus e}(t)+g_{\M/e}(t)+t\,g_{\N_{e,v}}(t).
\]
The matroids $\M\setminus e$, $\M/e$, and $\N_{e,v}$ are connected
and split, with $n-1$, $n-1$, and $n-2$ elements, respectively.
Each therefore satisfies the induction hypothesis, which concerns all
connected split matroids with fewer than $n$ elements. Their
$g$-polynomials have nonnegative coefficients, and the displayed
identity gives the same conclusion for $g_\M(t)$.

Let $\M_1,\ldots,\M_\ell$ be the connected components of a
nonempty split matroid $\M$ without loops or coloops. Then
$\M=\M_1\oplus\cdots\oplus\M_\ell$. Since split matroids are
closed under taking minors, each $\M_j$ is a connected split matroid.
The connected case
and the multiplicativity of the $g$-polynomial under direct sums give
\[
 g_\M(t)=\prod_{j=1}^{\ell}g_{\M_j}(t)\in\mathbb Z_{\ge0}[t].
\]
This completes the proof.
\end{proof}

\subsection{The deletion--contraction identity for paving matroids}
\label{subsec:paving}
Following Wang \cite[Section~2]{Wang2026}, for integers
$2\le q\le m$, write
\[
 p_{q,m}(t):=g_{\U_{q,m}}(t)+t g_{\U_{q-1,m-1}}(t).
\]
In particular, $p_{q,q}(t)=0$, since both uniform matroids in its
definition have coloops. The correction term below is the one appearing
in Ferroni and Schr\"oter \cite[Theorem~9.18 and its proof]{FerroniSchroter2024}.
For $2\le k\le h\le n-2$, we have
\begin{equation}\label{eq:paving-correction}
 \cC_{k-1,k,h,n}(t)=g_{\U_{k,h+1}}(t)+t g_{\U_{k-1,h}}(t)=p_{k,h+1}(t).
\end{equation}
For completeness, we verify this specialization with Delannoy paths.
Write $a=n-h-1$ and $b=h-k+1$.
The upper path in the lattice path presentation of $\La_{k-1,k,h,n}$ is
$\Nstep\Estep^a\Nstep^{k-1}\Estep^b$.
Every admissible Delannoy path for this upper path begins with $a$
east steps, from $(1,1)$ to $(a+1,1)$. Removing these steps and
translating by $(-a,0)$ gives a bijection with the admissible Delannoy
paths for $\Nstep^k\Estep^b$, the upper path for $\U_{k,h+1}$.
This bijection preserves the number of diagonal steps. The formula
of Ferroni \cite[Theorem~3.4]{Ferroni2023} for the $g$-polynomial
in terms of admissible Delannoy paths therefore gives
$g_{\La_{k-1,k,h,n}}(t)=g_{\U_{k,h+1}}(t)$.
Since $n-h\ge2$, the product of uniform matroid $g$-polynomials in
\eqref{eq:correction-definition} is
$g_{\U_{k-1,h}}(t)g_{\U_{1,n-h}}(t)=t g_{\U_{k-1,h}}(t)$.
This proves \eqref{eq:paving-correction}.

Now let $\M$ be a connected paving matroid of rank $k\ge3$ on an
$n$-element ground set $E$, where $n\ge4$, and suppose that $e\in E$ is such that
both $\M\setminus e$ and $\M/e$ are connected. Since every paving
matroid is split, Theorem~\ref{thm:split-recurrence} applies.
Fix $v\in E\setminus\{e\}$ and a total order on $E$, and let
$\N_{e,v}$ be the resulting auxiliary matroid.
The proper cyclic flats of $\M$ are precisely its hyperplanes
of cardinality at least $k$. Hence $r_F=k-1$ and $p_F=1$ for every
$F\in\mathcal Z^{\circ}(\M)$. It follows that $\mathcal O_e$ is empty
and that $\mathcal I_e$ consists of the proper cyclic flats containing
$e$ with cardinality at least $k+1$.
Every proper cyclic flat also has cardinality at most $n-2$:
the unique element outside a hyperplane of size $n-1$ would be a coloop,
contrary to connectedness.
Equations~\eqref{eq:Nv-polynomial} and~\eqref{eq:paving-correction} give
\begin{equation}\label{eq:paving-remainder}
 g_{\N_{e,v}}(t)
 =g_{\U_{k-1,n-2}}(t)
 -\sum_{\substack{F\in\mathcal Z^{\circ}(\M)\\e\in F}}
   p_{k-1,|F|-1}(t).
\end{equation}
Here we have included the terms with $|F|=k$, since
$p_{k-1,k-1}(t)=0$.
Substituting \eqref{eq:paving-remainder} into
\eqref{eq:intro-splitting} gives a recurrence that agrees with
the identity in Wang \cite[Proposition~2.5]{Wang2026}:
\[
 g_\M(t)=g_{\M\setminus e}(t)+g_{\M/e}(t)
 +t\left(
 g_{\U_{k-1,n-2}}(t)
 -\sum_{\substack{F\in\mathcal Z^{\circ}(\M)\\e\in F}}
 p_{k-1,|F|-1}(t)
 \right).
\]
Moreover, after the substitution
\eqref{eq:paving-correction}, the identities
\eqref{eq:internal-recurrence} and~\eqref{eq:internal-boundary} give,
respectively, the cases $h\ge k+1$ and $h=k$ of the recurrence for
the polynomials $p_{k,h+1}(t)$ stated by Wang
\cite[Lemma~2.4, equation~(2.3)]{Wang2026}.

The construction of $\N_{e,v}$ shows that it is paving.
Indeed, it has rank $k-1$, and, since $\mathcal O_e$ is empty,
each of its proper cyclic flats has rank $r_F-1=k-2$ for some
$F\in\mathcal I_e$. If $\N_{e,v}$ had a circuit of
cardinality at most $k-2$, the closure of that circuit would be a
proper cyclic flat of rank at most $k-3$, a contradiction.
The polynomial in parentheses is therefore the
$g$-polynomial of a connected paving matroid of rank $k-1$ on
$n-2$ elements.
\section*{Acknowledgements}
The authors contributed equally, and their names are listed in
alphabetical order by family name.
Alice L.~L.~Gao was partially supported by the National Natural Science
Foundation of China (Grant No.~12671393).
Matthew H.~Y.~Xie was partially supported by the National Natural Science
Foundation of China (Grant No.~12271403).

\section*{Declaration of Generative AI}
During the preparation of this manuscript, the authors used generative AI
tools to assist in exploring possible approaches, checking technical
details, and improving the exposition. All mathematical arguments,
computations, proofs, and results were independently verified by the
authors. The authors take full responsibility for the content of the
manuscript.

\end{document}